\documentclass[10pt]{amsart}
\usepackage{amsmath}
\usepackage[all]{xy}
\usepackage{url}
\usepackage{amsthm}
\usepackage{amssymb}
\newcommand{\gal}[2]{\ensuremath{\mathrm{Gal}(#1/#2)}}
\newcommand{\Ker}{\ensuremath{\mathrm{Ker}}}
\newcommand{\cue}{\ensuremath{\mathbb{Q}}}
\newcommand{\eff}{\ensuremath{\mathbb{F}}}
\newcommand{\zee}{\ensuremath{\mathbb{Z}}}

\newcommand{\gl}[3]{\ensuremath{\mathrm{GL}_{#1}(#2_{#3})}}
\newcommand{\tor}{\ensuremath{\mathrm{tor}}}
\newcommand{\zmod}[1]{\ensuremath{\zee/#1\zee}}
\newtheorem*{theorem*}{Theorem}
\newtheorem{theorem}{Theorem}
\newtheorem{proposition}[theorem]{Proposition}
\newtheorem{corollary}[theorem]{Corollary}

\newtheorem*{conjecture*}{Conjecture}
\newtheorem*{lemma*}{Lemma}
\newtheorem{lemma}[theorem]{Lemma}
\newtheorem*{claim*}{Claim}
\newtheorem*{diamond*}{Diamond Theorem}
\theoremstyle{remark}
\newtheorem{rem}[theorem]{Remark}
\newtheorem{definition}[theorem]{Definition}

\begin{document}

\title{Abelian Varieties and Galois Extensions of Hilbertian Fields}
\author{Christopher Thornhill}
\address{Department of Mathematics\\Indiana University\\Bloomington, IN 47405\\U.S.A.}
\email{cthornhi@indiana.edu}
\thanks{I would like to thank my advisor Michael J. Larsen of Indiana University for introducing me to this problem and providing much valuable guidance and feedback in producing this work. I would also like to thank Moshe Jarden for helping to improve this paper after I had first written it. He located some errors in the paper and provided very helpful feedback. He also graciously provided Lemmas \ref{lem7} and \ref{lem8} and alternative arguments for Lemmas \ref{lem1} and \ref{lem2}, Proposition \ref{prop2}, and Theorem \ref{thm5}, all of which cleared up the original presentation.}
%\subjclass[2000]{Primary 14K; Secondary 12E25, 12E30}

\begin{abstract} In a recent paper \cite{J09}, Moshe Jarden proposed a conjecture, later named the Kuykian conjecture, which states that if $A$ is an abelian variety defined over a Hilbertian field $K$, then every intermediate field of $K(A_{\tor})/K$ is Hilbertian. We prove that the conjecture holds for Galois extensions of $K$ in $K(A_{\tor})$.
\end{abstract}

\maketitle

\section{Introduction}
In his article ``Diamonds in Torsion of Abelian Varieties", Moshe Jarden made the following conjecture:

\begin{conjecture*} \cite{J09} Let $K$ be a Hilbertian field, $A$ an Abelian variety defined over $K$, and $M$ an extension of $K$ in $K(A_{\tor})$. Then $M$ is Hilbertian.
\end{conjecture*}

In the same article he proved the conjecture is true if $K$ is a number field, and in a later paper written by Fehm, Jarden, and Petersen \cite{FJP10}, the class of Hilbertian fields $K$ for which the conjecture holds was greatly extended. In each case, specific properties of the fields considered had to be used to verify the conjecture.

In this paper, we apply a group-theoretic approach which enables us to prove Jarden's conjecture for all Hilbertian fields provided that $M/K$ is a Galois extension. This approach requires the introduction of a special type of group, the \emph{Galois-Hilbertian} group.

\begin{definition}
A profinite group $G$ is called \emph{Galois-Hilbertian} if for every closed normal subgroup $H$ of $G$ the following property holds: If $K$ is a Hilbertian field and $L/K$ is a $G/H$-Galois extension, then $L$ is Hilbertian.
\end{definition}

\begin{rem} Suppose $G$ is Galois-Hilbertian and $H$ is a closed normal subgroup of $G$. Since each quotient of $G/H$ is also a quotient of $G$, it follows that $G/H$ is also Galois-Hilbertian. Thus, if $K$ is a Hilbertian field and $L/K$ is a $G$-Galois extension, then $L^{H}$ is a Hilbertian field. We see then that \emph{every} Galois extension of $K$ in $L$ is Hilbertian.
\end{rem}

There are several well-known examples of Galois-Hilbertian groups:
\begin{enumerate}
\item  Finite groups are Galois-Hilbertian since every finite extension of a Hilbertian field is Hilbertian. \cite[Prop. 12.3.5]{FrJ05}
\item  Abelian groups are Galois-Hilbertian since every abelian extension of a Hilbertian field is Hilbertian. This result was originally due to Kuyk. \cite[Thm. 16.11.3]{FrJ05}
\item  A profinite group is \emph{small} if for each positive integer $n$, the group has only finitely many open subgroups of index $n$ \cite[p. 328]{FrJ05}. Small groups are Galois-Hilbertian since every quotient of a small group is small and every Galois extension of a Hilbertian field with a small Galois group is Hilbertian \cite[Rmk. 16.10.3(d), Prop. 16.11.1]{FrJ05}. In fact, if $L$ is a Galois extension of a Hilbertian field $K$ with $\gal{L}{K}$ small, then any extension of $K$ in $L$ is Hilbertian \cite[Lemma 4]{J09}.

\end{enumerate}

The main result of this paper is the following:

\begin{theorem*}
For each n every closed subgroup of $\prod_{p}\gl{n}{\zee}{p}$ is Galois-Hilbertian.
\end{theorem*}

It is known that the maximal separable extension of $K$ in $K(A_{\tor})$ has a Galois group over $K$ which is a closed subgroup of $\prod_{p}\gl{2\dim(A)}{\zee}{p}$, so we can conclude that every Galois extension of a Hilbertian field $K$ in $K(A_{\tor})$ is Hilbertian.

The strategy employed is to show that extensions of Galois-Hilbertian groups by Galois-Hilbertian groups are again Galois-Hilbertian, and then we will see that any closed subgroup of $\prod_{p}\gl{n}{\zee}{p}$ can be expressed as an extension of a Galois-Hilbertian group by a Galois-Hilbertian group. The latter result utilizes a theorem of Larsen and Pink \cite{LP11} which describes special properties of subgroups of $\gl{n}{\eff}{p}$. We also state and apply several results regarding properties of $p$-adic analytic groups and products of finite simple groups, and we make use of Haran's diamond theorem to make several important reductions.

Throughout this paper we will consider many homomorphisms between profinite topological groups. In particular, we will use the following homomorphisms, all of which are continuous:

\begin{itemize}
\item  Inclusion of a closed subgroup into a group
\item  Projection from a closed subgroup of a direct product of groups to any subproduct of the groups
\item  The canonical isomorphisms from the first, second, and third isomorphism theorems
\item  The canonical isomorphism $\prod_{i}G_{i}\big/\prod_{i}N_{i}\cong \prod_{i}(G_{i}/N_{i})$, where each $N_{i}$ is normal in $G_{i}$

\end{itemize}

Since all of the groups we consider are closed subgroups of profinite groups, they are compact. Thus, we are only considering continuous maps from compact spaces to Hausdorff spaces, and so the homomorphisms are closed maps. We will use this fact often when we consider a closed subgroup of one group as a closed subgroup of another group by means of one or more of the homomorphisms listed above.

\section{Results}

\begin{proposition}  Every Galois-Hilbertian extension of a Galois-Hilbertian group is Galois-Hilbertian.\label{prop1}
\end{proposition}

\begin{proof} Suppose we have a closed normal subgroup $G'$ of $G$ such that $G'$ and $G/G'$ are Galois-Hilbertian. Let $H$ be any closed normal subgroup of $G$. The groups $G'/(G'\cap H)$ and $(G/G')\big/(G'H/G')$ are both Galois-Hilbertian since they are quotients of Galois-Hilbertian groups. We also have the canonical isomorphisms $$G'/(G'\cap H)\cong G'H/H$$ and $$(G/G')\big/(G'H/G')\cong G/G'H\cong(G/H)\big/(G'H/H).$$ Let $\bar{G'}=G'H/H$ and $\overline{G/H}=(G/H)\big/(G'H/H)$. We now view $G/H$ as an extension of the Galois-Hilbertian group $\overline{G/H}$ by the Galois-Hilbertian group $\bar{G'}$. Now let $K$ be a Hilbertian field and $L$ an extension of $K$ with $G/H=\gal{L}{K}$. Then $L^{\bar{G'}}$ is Hilbertian since $K$ is Hilbertian and $\overline{G/H}\cong\gal{L^{\bar{G'}}}{K}$ is Galois-Hilbertian. Also, $L$ is Hilbertian since $L^{\bar{G'}}$ is Hilbertian and $\bar{G'}\cong\gal{L}{L^{\bar{G'}}}$ is Galois-Hilbertian. Thus, $G$ is Galois-Hilbertian. \end{proof}

We now introduce a class of groups called \emph{$k$-stage} groups. We prove that $k$-stage groups are Galois-Hilbertian and then use them to construct other Galois-Hilbertian groups.

\begin{definition}
A topological group $G$ is called \emph{1-stage} if it is trivial or a direct product of finite simple groups. A topological group $G$ is called \emph{k-stage} for $k\geq2$ if it has a closed normal subgroup $G'$ such that $G'$ is 1-stage and $G/G'$ is $(k-1)$-stage.
\end{definition}

\begin{rem}If a group $G$ is $k$-stage, then it is $(k+1)$-stage, for $G$ always has the subgroup $\{e\}$, and $G/\{e\}$ is $k$-stage. Therefore, $G$ is also $j$-stage for any $j\geq k$.
\end{rem}

We now establish that certain closed subgroups of $k$-stage groups are also $k$-stage groups, but first we will require several results regarding closed subgroups of direct products of finite simple groups.

\begin{lemma}
Let $G=\prod_{i\in I}S_{i}$ be a direct product of finitely many finite simple groups, and let $H$ be a subgroup of $G$. Suppose the projection of $H$ on each of the factors of $G$ is surjective. Then there is a subset J of I such that $H\cong\prod_{j\in J}S_{j}$.\label{lem7}
\end{lemma}
\begin{proof} We proceed by induction. The result is trivial if $|I|=1$, so suppose that the result holds for any direct product of $k$ finite simple groups for some $k\geq1$, and let $|I|=k+1$. Choose some $i\in I$, set $I'=I\smallsetminus\{i\}$, and let $G'=\prod_{i'\neq i}S_{i'}$, so $G=S_{i}\times G'$. Let $\pi_{i}:G\rightarrow S_{i}$ and $\pi':G\rightarrow G'$ be the projection maps. Then $H'=\pi'(H)$ is a subgroup of $G'$ whose projection on each of the factors of $G'$ is surjective. By induction, there is some $J'\subset I'$ such that \begin{equation}\label{eq1}
H'\cong\prod_{j\in J'}S_{j}. \end{equation} The map $h \mapsto (\pi_{i}(h),\pi'(h))$ embeds $H$ into $S_{i}\times H'$, so $H$ is a subgroup of $S_{i}\times H'$ which projects surjectively to each factor. Let $N'=H\cap\Ker(\pi_{i})$ and $N_{i}=H\cap\Ker(\pi')$. By Goursat's Lemma, $H'/N'\cong S_{i}/N_{i}$. Since $S_{i}$ is simple, either $N_{i}$ is trivial or $N_{i}=S_{i}$. If $N_{i}$ is trivial, then $H\cong H'$. If $N_{i}=S_{i}$, then $H'=N'$, so $H=S_{i}\times H'$. In both cases, the conclusion of the lemma follows from (\ref{eq1}). \end{proof}

\begin{lemma}
Let $(H_{x},\alpha_{yx})_{x,y\in X}$ be a projective system of finite groups such that the homomorphism $\alpha_{yx}:H_{y}\rightarrow H_{x}$ is surjective for all $y\geq x$. For each $x\in X$ we assume that $H_{x}=\prod_{i\in I_{x}}S_{i}$ is a direct product of finitely many finite non-abelian simple groups. Then $H=\varprojlim H_{x}$ is a direct product of non-abelian simple groups, each isomorphic to a group belonging to the set $\bigcup_{x\in X}\{S_{i}\mid i\in I_{x}\}$.\label{lem8}
\end{lemma}
\begin{proof}
We may assume that $0\not\in I_{x}$ and set $I'_{x}={0}\cup I_{x}$. Given $x\leq y$ in $X$, we define a map $\beta_{yx}:I'_{y}\rightarrow I'_{x}$ in the following way. First we set $\beta_{yx}(0)=0$. Next let $j\in I_{y}$. Then either $\alpha_{yx}(S_{j})$ is trivial or $\alpha_{yx}(S_{j})$ is a non-abelian simple subgroup of $H_{x}$. In the former case we set $\beta_{yx}(j)=0$. In the latter case there exists a unique $i\in I_{x}$ such that $\alpha_{yx}(S_{j})=S_{i}$ \cite[p. 51, Satz 9.12(b)]{Hup67}. We set $\beta_{yx}(j)=i$ and note that $\alpha_{yx}$ maps $S_{j}$ isomorphically onto $S_{i}$.

Now let $I^{0}_{yx}=\{j\in I_{y}\mid\beta_{yx}(j)=0\}$ and $I_{yx}=I_{y}\smallsetminus I^{0}_{yx}$. Then $\beta_{yx}$ maps $I_{yx}$ bijectively onto $I_{x}$. Also, $\prod_{j\in I^{0}_{yx}}S_{j}=\Ker(\alpha_{yx})$ and $\alpha_{yx}$ maps $\prod_{j\in I_{yx}}S_{j}$ isomorphically onto $H_{x}=\prod_{i\in I_{x}}S_{i}$.

If $z\in X$ and $z\geq y$, then the uniqueness in the first paragraph of the proof implies that $\beta_{yx}\circ\beta_{zy}=\beta_{zx}$. Moreover, $\beta_{xx}:I'_{x}\rightarrow I'_{x}$ is the identity map. It follows that $(I'_{x},\beta_{yx})_{x,y\in X}$ is a projective system of finite sets. Let $I'=\varprojlim I'_{x}$ and $I=I'\smallsetminus\{0\}$. Thus, $I'$ is a profinite space and $I$ is an open subset of $I'$.

For each $x\in X$ let $\beta_{x}:I'\rightarrow I'_{x}$ be the inverse limit of the maps $\beta_{yx}:I'_{y}\rightarrow I'_{x}$ with $y\geq x$. Also, let $I^{x}=\varprojlim_{y\geq x}I_{yx}$. Then $I^{x}$ is a finite subset of $I$ and $\beta_{x}$ maps $I^{x}$ bijectively onto $I_{x}$ and $\beta_{x}(I'\smallsetminus I^{x})=\{0\}$. If $y\geq x$, then $I^{x}\subset I^{y}$.

Again, for each $x\in X$ let $\alpha_{x}:H\rightarrow H_{x}$ be the inverse limit of the homomorphisms $\alpha_{yx}:H_{y}\rightarrow H_{x}$ with $y\geq x$. For each $i\in I^{x}$ we set $i_{x}=\beta_{x}(i)$. Set $S_{i}=\varprojlim_{y\geq x}S_{i_{y}}$. Since each of the maps $\alpha_{yx}:S_{i_{y}}\rightarrow S_{i_{x}}$ is an isomorphism, so is the map $\alpha_{x}:S_{i}\rightarrow S_{i_{x}}$. In particular, $S_{i}$ is a finite simple non-abelian subgroup of $H$. Moreover, $S_{i}$ is normal in $H$ because $S_{i_{x}}$ is normal in $H_{x}$ for each $i\in I^{x}$.

It follows that $\alpha_{x}$ maps $\langle S_{i}\mid i\in I^{x}\rangle$ isomorphically onto $H_{x}=\prod_{i\in I_{x}}S_{i}$. Thus, $\langle S_{i}\mid i\in I^{x}\rangle=\prod_{i\in I^{x}}S_{i}$. Since for each $x\in X$ the group $H_{x}$ is generated by the groups $S_{i}$ with $i\in I_{x}$, the group $H$ is generated by the groups $S_{i}$, for $i\in I$. It follows that $H=\prod_{i\in I}S_{i}$, as claimed. \end{proof}

\begin{lemma} Let $G=\prod_{i\in I}S_{i}$ be a direct product of finite non-abelian simple groups, and let $H$ be a closed subgroup of $G$. Suppose that the projection of $H$ to each of the factors of $G$ is surjective. Then $H$ is isomorphic to a direct product of simple groups, each belonging to the set $\{S_{i}\mid i\in I\}$.\label{lem1}
\end{lemma}
\begin{proof}
For each finite subset $J$ of $I$ let $G_{J}=\prod_{j\in J}S_{j}$ and let $H_{J}$ be the projection of $H$ to $G_{J}$ . Lemma \ref{lem7} gives a subset $J'$ of $J$ such that $H_{J}\cong\prod_{j\in J'}S_{j}$. Thus, for each $j\in J'$, there is a subgroup $S'_{j}$ of $H_{j}$ which is isomorphic to $S_{J}$ such that $H_{J}=\prod_{j\in J'}S'_{j}$. By Lemma \ref{lem8}, $H=\varprojlim H_{J}$ is a direct product of non-abelian simple groups, each isomorphic to a group belonging to the set $\{S_{i}\mid i\in I\}$.
\end{proof}

\begin{lemma} For each prime $p$, let $G_{p}$ be a pro-$p$ group, and let $H$ be a closed subgroup of $\prod_{p}G_{p}$ such that the projection $\pi_{p}:H \rightarrow G_{p}$ is surjective for each $p$. Then $H=\prod_{p}G_{p}$.\label{lem2}
\end{lemma}

\begin{proof} If $G$ is finite, then $|G|=\prod_{p}p^{n_{p}}$. Since $H/\Ker(\pi_{p})\cong G_{p}$, we have that $p^{n_{p}}$ divides $|H|$ for each $p$. Hence, $|H|=|G|$, so $H=G$.

If $G$ is an infinite group, then consider an open normal subgroup $N$ of the form $\prod_{p}N_{p}$, where $N_{p}$ is an open normal subgroup of $G_{p}$ for each $p$ and $N_{p}=G_{p}$ for almost all $p$. Then $G/N\cong\prod_{p}G_{p}/N_{p}$, and the projection of $HN/N$ onto $G_{p}/N_{p}$ is surjective for each $p$. Since $N_{p}$ has finite index in $G_{p}$ for each $p$, $G/N$ is finite. Thus, by the previous paragraph, $HN/N=G/N$; hence, $HN=G$. Since the intersection of all the $N$ as above is trivial and $H$ is a closed subgroup of $G$, we have that $H=G$ \cite[Lemma 1.2.2(b)]{FrJ05}. \end{proof}

\begin{lemma} Let $H$ be a closed subgroup of $G_{1}\times G_{2}$, where $G_{1}$ is a direct product of non-abelian finite simple groups and $G_{2}$ is abelian. If the projection of $H$ to each of $G_{1}$ and $G_{2}$ is surjective, then $H=G_{1}\times G_{2}$.\label{lem3}

\end{lemma}

\begin{proof}

Let $H$, $G_{1}$, and $G_{2}$ be as above, and let $\pi_{1}:H\rightarrow G_{1}$ and $\pi_{2}:H\rightarrow G_{2}$ be the projection maps, which are surjective. Thus, $\Ker(\pi_{1})$ and $\Ker(\pi_{2})$ are normal subgroups of $G_{2}$ and $G_{1}$, respectively. It is known that every closed normal subgroup of a direct product of finite simple non-abelian groups is itself a direct product of a subcollection of those groups \cite[Lemma 18.3.9]{FrJ05}, so $G_{1}\big/\Ker(\pi_{2})$ is isomorphic to a direct product of non-abelian finite simple groups, which is necessarily non-abelian unless $\Ker(\pi_{2})=G_{1}$. By Goursat's Lemma, $$G_{1}\big/\Ker(\pi_{2})\cong G_{2}\big/\Ker(\pi_{1}).$$ Since $G_{2}\big/\Ker(\pi_{1})$ is abelian, it follows that $\Ker(\pi_{2})=G_{1}$ and $\Ker(\pi_{1})=G_{2}$. Thus, $H=G_{1}\times G_{2}$. \end{proof}

\begin{lemma} Let $H$ be a closed subgroup of a direct product of finite simple groups which projects surjectively to each factor. Then $H$ is a direct product of finite simple groups.\label{lem4}
\end{lemma}

\begin{proof}

Let $H$ be a closed subgroup of $G=\prod_{i}G_{i}$, where the $G_{i}$ are finite simple groups. Let $A$ be the set of indices $i$ for which $G_{i}$ is abelian, and let $R$ be the set defined in Lemma \ref{lem1}, where the equivalence relation $\sim$ is only defined on the $G_{i}$ which are non-abelian. Let $\pi_{A}:H\rightarrow\prod_{i\in A}G_{i}$ and $\pi_{R}:H\rightarrow\prod_{j\in R}G_{j}$ be the projection maps. First we write $\prod_{i\in A}G_{i}$ as $\prod_{p}(\zmod{p})^{\alpha_{p}}$, where $\alpha_{p}$ is a (possibly infinite) cardinal, by grouping all of the abelian finite simple groups of order $p$ together for each prime $p$, and we let $\pi_{p}:H\rightarrow (\zmod{p})^{\alpha_{p}}$ be the projection map. Then $\pi_{A}(H)$ is a closed subgroup of the direct product of the pro-$p$ groups $\pi_{p}(H)$. From Lemma \ref{lem2} we get that $\pi_{A}(H)=\prod_{p}\pi_{p}(H)$. In addition, $\pi_{p}(H)$ is a closed subgroup of a $\zmod{p}$-vector space, so it is isomorphic to $(\zmod{p})^{\beta_{p}}$ for some $\beta_{p}\leq\alpha_{p}$ \cite[Lemma 22.7.3]{FrJ05}. Hence, $\pi_{A}(H)$ is a direct product of finite simple groups. From Lemma \ref{lem1} we also have that $\pi_{R}(H)=\prod_{j\in R}G_{j}$. Thus, we can view $H$ as a closed subgroup of $\prod_{j\in R}G_{j}\times\prod_{p}\pi_{p}(H)$ which projects surjectively to each of the two factors. We conclude from Lemma \ref{lem3} that $H=\prod_{j\in R}G_{j}\times\prod_{p}\pi_{p}(H)$, and so we have that $H$ is a direct product of finite simple groups. \end{proof}

\begin{proposition}  Let $H$ be a closed subgroup of a direct product of $k$-stage groups such that the projection of $H$ to each factor is surjective. Then $H$ is $k$-stage.\label{prop2}
\end{proposition}

\begin{proof}  We proceed by induction on $k$. A direct product of 1-stage groups is again a 1-stage group, so suppose $H$ is a closed subgroup of a 1-stage group which projects surjectively to each factor. From Lemma \ref{lem4}, we see that $H$ is isomorphic to a product of finite simple groups, so $H$ is 1-stage.

Now suppose that $H$ is a closed subgroup of a product of $(k+1)$-stage groups $G=\prod_{i} G_{i}$ which projects surjectively to each factor, and any closed subgroup of a product of $k$-stage groups which projects surjectively to each of the factors is again a $k$-stage group. For each $G_{i}$ we have a closed normal subgroup $G'_{i}$ such that $G'_{i}$ is 1-stage and $\bar{G}_{i}=G_{i}/G'_{i}$ is $k$-stage.

Let $G'=\prod_{i} G'_{i}$, so $H'=H\cap G'$ is a closed normal subgroup of $H$. We claim that $H'$ is 1-stage. Let $\pi_{i}:H\rightarrow G_{i}$ be the projection map. Then $G''_{i}=\pi_{i}(H')$ is a closed normal subgroup of $G_{i}$, hence also of $G'_{i}$. Since $G'_{i}$ is a direct product of finite simple groups, it follows that $G''_{i}$ is also a direct product of finite simple groups \cite[Lemma 25.5.3(b)]{FrJ05}. Now we have that $G''=\prod_{i}G''_{i}$ is a 1-stage group, and $H'$ is a closed subgroup of $G''$ which projects surjectively to each factor. From the first paragraph of the proof, $H'$ is 1-stage.

Note that $H/H'\cong HG'/G'$, which is a closed subgroup of $G/G'\cong\prod_{i}\bar{G}_{i}$. Since $\pi(H)=G_{i}$ and $\pi_{i}(H')=G''_{i}$, the projection of $HG'/G'$ to the $i$th component is $G_{i}/G''_{i}=\bar{G}_{i}$. By induction, $HG'/G'$ is $k$-stage, so $H/H'$ is also $k$-stage. Therefore, $H$ is a $(k+1)$-stage group. \end{proof}

\begin{lemma} Every finite group $G$ is $k$-stage, where $k\leq |G|$.\label{lem5}
\end{lemma}

\begin{proof} The result is trivial if $|G|=1$, so suppose that all groups of order at most $k$ are $k$-stage. Let $G$ be a finite group with $|G|=k+1$. If $G$ is simple, then $G$ is 1-stage, so suppose $G$ is not simple. Then $G$ has a minimal normal subgroup $H$, and so $H\cong \prod_{i=1}^{n}S$ for some finite simple group $S$; hence, $H$ is 1-stage. Also, $|G/H|\leq k$, so $G/H$ is $k$-stage. Therefore, $G$ is $(k+1)$-stage. \end{proof}

\begin{corollary}  Let $H$ be a closed subgroup of a direct product of groups of order bounded by $k$ such that the projection of $H$ to each factor is surjective.  Then $H$ is $k$-stage.\label{cor1}
\end{corollary}

\begin{proof} Suppose each of the groups in the product has order bounded by $k$. By Lemma \ref{lem5}, all of these groups are $k$-stage. Thus, $H$ is a closed subgroup of a direct product of $k$-stage groups which projects surjectively to each factor, so by Proposition \ref{prop2}, $H$ is $k$-stage. \end{proof}

Haran's diamond theorem is a powerful result for identifying Hilbertian fields within Galois extensions of Hilbertian fields. We will need this result to make several important reductions.

\begin{diamond*}
\cite[Thm. 13.8.3]{FrJ05} Let $K$ be a Hilbertian field, $M_{1}$ and $M_{2}$ Galois extensions of $K$, and $M$ an intermediate field of $M_{1}M_{2}/K$. Suppose that $M\not\subset M_{1}$ and $M\not\subset M_{2}$. Then $M$ is Hilbertian.
\end{diamond*}

\begin{lemma}  Let $\{G_{i}\}_{i\in I}$ be a collection of groups with the property that whenever $F$ is a Hilbertian field and $M/F$ is a Galois extension with $\gal{M}{F}\cong G_{i}$ for some $i$, then every intermediate extension of $M/F$ is Hilbertian. Suppose $K$ is a Hilbertian field and $L/K$ is a Galois extension with $\gal{L}{K}\cong\prod_{i\in I}G_{i}$. Then every intermediate extension of $L/K$ is Hilbertian.\label{lem6}
\end{lemma}

\begin{proof}  Let $K$, $L$, and $\{G_{i}\}_{i\in I}$ be as above, and let $K'$ be an extension of $K$ in $L$, so $K'=L^{H}$ for some closed subgroup $H$ of $\prod_{i\in I}G_{i}$. Suppose there are two indices $i_{1}$ and $i_{2}$ such that $H$ contains neither $G_{i_{1}}$ nor $G_{i_{2}}$. Let $G_{1}'=\prod_{i\neq i_{1}}G_{i}$, and then $L^{G_{i_{1}}}L^{G_{1}'}=L$, and $K'$ is an extension of $K$ which is contained in neither $L^{G_{i_{1}}}$ nor $L^{G_{1}'}$.  Thus, by Haran's diamond theorem, $K'$ is Hilbertian. If, on the other hand, $H$ contains all but possibly one $G_{i}$, say $G_{i_{1}}$, then $H$ contains the subgroup $G_{1}'$ above. Thus, $K'\subset L^{G_{1}'}$. Since $\gal{L^{G_{1}'}}{K}\cong G_{i_{1}}$, we have that $K'$ is Hilbertian. \end{proof}

\begin{proposition}  Every $k$-stage group is Galois-Hilbertian.\label{prop3}
\end{proposition}

\begin{proof}  We proceed by induction on $k$. If $G$ is a 1-stage group and $H$ is any closed normal subgroup of $G$, then $G/H$ is a direct product of finite simple groups \cite[Lemma 25.5.3(d)]{FrJ05}. Finite groups satisfy the hypotheses of the groups in Lemma \ref{lem6}, so in particular we see that every $G/H$-extension of a Hilbertian field is Hilbertian. Thus, $G$ is Galois-Hilbertian.

Now suppose $G$ is $(k+1)$-stage and all $k$-stage groups are Galois-Hilbertian, so $G$ has a closed normal subgroup $G'$ such that $G'$ is 1-stage and $G/G'$ is $k$-stage. Thus, $G'$ and $G/G'$ are Galois-Hilbertian. It follows from Proposition \ref{prop1} that $G$ is Galois-Hilbertian. \end{proof}

\begin{lemma}  Every abelian profinite extension of a $k$-stage group is Galois-Hilbertian.\label{thm1}
\end{lemma}

\begin{proof}  Abelian groups and $k$-stage groups are Galois-Hilbertian, so the result follows from Proposition \ref{prop1}. \end{proof}

A group $G$ is called \emph{$p$-adic analytic} if $G$ is an analytic manifold over $\cue_{p}$ such that the functions $(x,y)\mapsto xy$ and $x\mapsto x^{-1}$ for $x$ and $y$ in $G$ are analytic. We will be particularly interested in $p$-adic analytic groups which are also pro-$p$ groups. There are many equivalent characterizations of pro-$p$ $p$-adic analytic groups. We will adopt one of these characterizations and say that a pro-$p$ group $G$ is $p$-adic analytic if it is isomorphic to a closed subgroup of $\gl{n}{\zee}{p}$ for some $n$ \cite[p. 97]{DDMS99}. There are several properties of $p$-adic analytic groups that will prove useful. In particular, every closed subgroup of $\gl{n}{\zee}{p}$ is finitely generated \cite[Lemma 22.14.4]{FrJ05}, so $p$-adic analytic groups are finitely generated; hence, they are small \cite[Lemma 16.10.2]{FrJ05}. Hence, closed subgroups and quotients of pro-$p$ $p$-adic analytic groups are pro-$p$ $p$-adic analytic groups \cite[Thm. 9.6]{DDMS99}. We now introduce another class of groups, which we term \emph{$\zee$-analytic}.

\begin{definition}
A topological group is called \emph{$\zee$-analytic} if it is isomorphic to $\prod_{p}G_{p}$, where for each prime $p$, $G_{p}$ is a pro-$p$ $p$-adic analytic group.
\end{definition}

\begin{proposition} Every closed subgroup of a $\zee$-analytic group is $\zee$-analytic.\label{prop4}
\end{proposition}

\begin{proof} Suppose $H$ is a closed subgroup of a $\zee$-analytic group $G=\prod_{p}G_{p}$. Let $\pi_{p}$ be the projection of $H$ to $G_{p}$. Then $\pi_{p}(H)$ is a closed subgroup of $G_{p}$, so $\pi_{p}(H)$ is a pro-$p$ $p$-adic analytic group. Thus, $H$ is a closed subgroup of the $\zee$-analytic group $\prod_{p}\pi_{p}(H)$. It follows from Lemma \ref{lem2} that $H=\prod_{p}\pi_{p}(H)$, so $H$ is $\zee$-analytic. \end{proof}

\begin{proposition} Every $\zee$-analytic group is Galois-Hilbertian.\label{prop5}
\end{proposition}

\begin{proof}
We use the notation from the proof of Proposition \ref{prop4}. Let $H$ be any closed normal subgroup of $G$, and then $H=\prod_{p}\pi_{p}(H)$ is a $\zee$-analytic group. Also, $G/H\cong \prod_{p}(G_{p}/\pi_{p}(H))$, and since quotients of pro-$p$ $p$-adic analytic groups by closed normal subgroups are again pro-$p$ $p$-adic analytic groups, we have that $G/H$ is $\zee$-analytic. Since $G_{p}/\pi_{p}(H)$ is a small group for each $p$, the groups $G_{p}/\pi_{p}(H)$ satisfy the hypotheses of the groups in Lemma \ref{lem6}. Thus, every $G/H$-extension of a Hilbertian field is Hilbertian, so $G$ is Galois-Hilbertian. \end{proof}

\begin{rem}
It is worth noting that Proposition \ref{prop4} would not necessarily be true if the groups $G_{p}$ were not pro-$p$ groups. In particular, this means that a closed subgroup of $\prod_{p}\gl{n}{\zee}{p}$ is not necessarily a direct product of closed subgroups of the $\gl{n}{\zee}{p}$. If this were the case, then the Kuykian conjecture would follow from a slight modification of Lemma \ref{lem6}. The proof of Proposition \ref{prop4} requires the result of Lemma \ref{lem2}, which depends on the fact that quotients of the factors in the direct product have orders that are pairwise relatively prime. This is not the case in general.
\end{rem}

\begin{lemma}  Every $\zee$-analytic extension of a Galois-Hilbertian group is Galois-Hilbertian.\label{thm2}
\end{lemma}

\begin{proof} $\zee$-analytic groups are Galois-Hilbertian, and so the result follows from Proposition \ref{prop1}. \end{proof}

\begin{corollary}  Every $\zee$-analytic extension of an abelian extension of a $k$-stage group is Galois-Hilbertian.\label{cor2}
\end{corollary}

\begin{proof} By Lemma \ref{thm1} we know that an abelian extension of a $k$-stage group is Galois-Hilbertian. Thus, by Lemma \ref{thm2} we have that a $\zee$-analytic extension of an abelian extension of a $k$-stage group is Galois-Hilbertian. \end{proof}

\begin{proposition}  For each $n$ there exists some $k$ such that every closed subgroup of $\prod_{p}\gl{n}{\eff}{p}$ is a $\zee$-analytic extension of an abelian extension of a $k$-stage group.\label{thm3}
\end{proposition}

\begin{proof} Let $H$ be a closed subgroup of $\prod_{p}\gl{n}{\eff}{p}$. Consider the subgroups $H_{p}=\pi_{p}(H)$, where $\pi_{p}$ is projection to the $p$th factor. We simultaneously view $H_{p}$ as a subgroup of $\gl{n}{\eff}{p}$ and $\prod_{p}\gl{n}{\eff}{p}$ in the natural way. Larsen and Pink showed that there are normal subgroups $\Gamma_{p,1}\supseteq\Gamma_{p,2}\supseteq\Gamma_{p,3}$ of $H_{p}$ satisfying the following properties: \cite[Thm. 0.2]{LP11}
\begin{enumerate}
\item  $[H_{p}:\Gamma_{p,1}]$ is bounded by a constant $J'(n)$ which depends only on $n$.
\item  $\Gamma_{p,1}/\Gamma_{p,2}$ is a direct product of finite simple groups.
\item  $\Gamma_{p,2}/\Gamma_{p,3}$ is abelian with order not divisible by $p$.
\item  $\Gamma_{p,3}$ is a $p$-group.
\end{enumerate}

We wish to show that there are closed normal subgroups $A_{1}$ of $H$ and $A_{2}$ of $H/A_{1}$ such that $A_{1}$ is $\zee$-analytic, $A_{2}$ is abelian, and $(H/A_{1})\big/A_{2}$ is $k$-stage.

First, we define $A_{1}$ as $$A_{1}=H\cap\prod_{p}\Gamma_{p,3}.$$ Each $\Gamma_{p,3}$ is a pro-$p$ $p$-adic analytic group, so $\prod_{p}\Gamma_{p,3}$ is $\zee$-analytic. From Proposition \ref{prop4} we have that every closed subgroup of a $\zee$-analytic group is $\zee$-analytic, so $A_{1}$ is $\zee$-analytic.

Now we let $B=\prod_{p}\Gamma_{p,2}$ and $$A_{2}=(H\cap B)\big/A_{1},$$ so $A_{2}$ is then a closed normal subgroup of $H/A_{1}$. Note that $$A_{2}=\frac{(H\cap B)}{(H\cap B)\cap\prod_{p}\Gamma_{p,3}}\cong\frac{(H\cap B)(\prod_{p}\Gamma_{p,3})}{\prod_{p}\Gamma_{p,3}}\leq\frac{B}{\prod_{p}\Gamma_{p,3}}\cong\prod_{p}(\Gamma_{p,2}/\Gamma_{p,3}).$$ Since the groups $\Gamma_{p,2}/\Gamma_{p,3}$ are abelian, so is $\prod_{p}(\Gamma_{p,2}/\Gamma_{p,3})$, and we see that $A_{2}$ is abelian.

It only remains to prove that $(H/A_{1})\big/A_{2}$ is $k$-stage. Since $H$ is a subgroup of $\prod_{p}H_{p}$ and $B$ is a normal subgroup of $\prod_{p}H_{p}$, we have that $$(H/A_{1})\big/A_{2}\cong H/(H\cap B)\cong HB\big/B.$$ Now $HB\big/B$ is a closed subgroup of the group $\prod_{p}H_{p}\big/B$, and the latter group is isomorphic to $\prod_{p}(H_{p}/\Gamma_{p,2})$. For each $p$, $H_{p}/\Gamma_{p,2}$ is an extension of the group $H_{p}/\Gamma_{p,1}$ by the group $\Gamma_{p,1}/\Gamma_{p,2}$, and $\Gamma_{p,1}/\Gamma_{p,2}$ is a direct product of finite simple groups. Thus, by definition, $H_{p}/\Gamma_{p,2}$ is $k$-stage for some $k\leq |H_{p}/\Gamma_{p,1}|+1\leq J'(n)+1$, so we can let $k=J'(n)+1$. By Proposition \ref{prop2}, $\prod_{p}(H_{p}/\Gamma_{p,2})$ is $k$-stage, and since $H$ projects surjectively to each of the factors of $\prod_{p}H_{p}$, we see that $HB\big/B$ projects surjectively to each of the factors of $\prod_{p}(H_{p}/\Gamma_{p,2})$. Thus, by Proposition \ref{prop2}, $HB\big/B$ is $
 k$-stage, and so $(H/A_{1})\big/A_{2}$ is $k$-stage. \end{proof}

An immediate consequence of Corollary \ref{cor2} and Proposition \ref{thm3} is the following result:

\begin{corollary} For each $n$ every closed subgroup of $\prod_{p}\gl{n}{\eff}{p}$ is Galois-Hilbertian.\label{cor3}
\end{corollary}

Now we are ready to prove the main result:

\begin{theorem}  For each $n$ every closed subgroup of $\prod_{p}\gl{n}{\zee}{p}$ is Galois-Hilbertian.\label{thm4}
\end{theorem}

\begin{proof} First we make a small reduction. If $H$ is a closed subgroup of $A\times B$, where $A=\gl{n}{\zee}{2}$ and $B=\prod_{p>2}\gl{n}{\zee}{p}$, then consider the projection $\pi_{B}:H\rightarrow B$. If $H_{0}=\Ker(\pi_{B})$, then $H_{0}$ is a closed subgroup of $\gl{n}{\zee}{2}$, so $H_{0}$ is finitely generated, hence, small. Thus, $H_{0}$ is Galois-Hilbertian. Now $\pi_{B}(H)$ is a closed subgroup of $\prod_{p>2}\gl{n}{\zee}{p}$, so if $\pi_{B}(H)$ is Galois-Hilbertian, then $H$ is Galois-Hilbertian by Proposition \ref{prop1}. Thus we need only consider closed subgroups of $\prod_{p>2}\gl{n}{\zee}{p}$.

Let $H$ be a closed subgroup of $\prod_{p>2}\gl{n}{\zee}{p}$, $K$ a Hilbertian field, and $L/K$ a Galois extension with $\gal{L}{K}\cong H$. For each prime $p>2$ let $$N_{p}=\{M\in\gl{n}{\zee}{p}:M\equiv I \pmod{p}\}$$ be the principal congruence subgroup. Let $N'=\prod_{p}N_{p}$ and let $N=H\cap N'$. Note that $$H/N\cong HN'/N'<\prod_{p>2}\gl{n}{\zee}{p}/N'\cong\prod_{p>2}\gl{n}{\eff}{p}.$$ Thus, we can view $H/N$ as a closed subgroup of $\prod_{p}\gl{n}{\eff}{p}$. From Corollary \ref{cor3} we see that $H/N$ is Galois-Hilbertian.

It is known that $N_{p}$ is a pro-$p$ $p$-adic analytic group for each $p>2$ \cite[Lemma 22.14.2]{FrJ05}, so $\prod_{p>2}N_{p}$ is a $\zee$-analytic group, and $N$ is a closed subgroup of $\prod_{p>2}N_{p}$. Thus, by Proposition \ref{prop4}, $N$ is also $\zee$-analytic. Now we see that $H$ is a $\zee$-analytic extension of the Galois-Hilbertian group $H/N$, so by Lemma \ref{thm2}, $H$ is Galois-Hilbertian. \end{proof}

Immediately, we have the following:

\begin{theorem} Let $K$ be a Hilbertian field and $A$ an abelian variety defined over $K$. If $M$ is a Galois extension of $K$ in $K(A_{tor})$, then $M$ is Hilbertian.\label{thm5}
\end{theorem}

\begin{proof} For each prime $p$, we have $A_{p^{\infty}}=\bigcup_{n=1}^{\infty}A_{p^{n}}$, where $A_{p^{n}}$ is the set of $p^{n}$-torsion points of $A$, and we have $A_{\tor}=\bigcup_{p}A_{p^{\infty}}$. Let $K_{p}$ be the maximal separable extension of $K$ in $K(A_{p^{\infty}})$, and let $K'$ be the compositum of the $K_{p}$. Since $K(A_{\tor})$ is the compositum of all the $K(A_{p^{\infty}})$, we have that $K'$ is the maximal separable extension of $K$ in $K(A_{\tor})$. Thus, if $M$ is a Galois extension of $K$ in $K(A_{\tor})$, then $M\subset K'$. For each $p$, $\gal{K_{p}}{K}$ is a closed subgroup of $\gl{2\dim(A)}{\zee}{p}$ (proof of \cite[Lemma 6]{J09}), so $\gal{K'}{K}$ is a closed subgroup of $\prod_{p}\gl{2\dim(A)}{\zee}{p}$. By Theorem \ref{thm4}, $\gal{K'}{K}$ is Galois-Hilbertian, so $M$ is Hilbertian. \end{proof}

\bibliographystyle{plain}
\bibliography{bibliography}

\begin{thebibliography}{1}

\bibitem{DDMS99}
J.~D. Dixon, M.~P.~F. du~Sautoy, A.~Mann, and D.~Segal.
\newblock {\em Analytic pro-{$p$} groups}, volume~61 of {\em Cambridge Studies
  in Advanced Mathematics}.
\newblock Cambridge University Press, Cambridge, second edition, 1999.

\bibitem{FJP10}
Arno Fehm, Moshe Jarden, and Sebastian Petersen.
\newblock Kuykian fields.
\newblock {\em Forum Mathematicum}, 2010.
\newblock \url{http://dx.doi.org/10.1515/FORM.2011.094}.

\bibitem{FrJ05}
Michael~D. Fried and Moshe Jarden.
\newblock {\em Field arithmetic}, volume~11 of {\em Ergebnisse der Mathematik
  und ihrer Grenzgebiete. 3. Folge. A Series of Modern Surveys in Mathematics}.
\newblock Springer-Verlag, Berlin, second edition, 2005.

\bibitem{Hup67}
B.~Huppert.
\newblock {\em Endliche Gruppen I}, volume 134 of {\em Die Grundlehren der
  mathematischen Wissenschaften in Einzeldarstellungen}.
\newblock Springer, Berlin, 1967.

\bibitem{J09}
Moshe Jarden.
\newblock Diamonds in torsion of abelian varieties.
\newblock {\em J. Inst. Math. Jussieu}, 9(3):477--480, 2010.

\bibitem{LP11}
Michael~J. Larsen and Richard Pink.
\newblock Finite subgroups of algebraic groups.
\newblock {\em J. Amer. Math. Soc.}, 2011.
\newblock (to appear).

\end{thebibliography}
\end{document}